\documentclass[12pt]{amsart}

\usepackage{amsthm,amsmath,amsfonts,amssymb}
\usepackage{color}
\usepackage{enumerate}
\usepackage{enumitem}
\usepackage{algorithmic}
\usepackage[T1]{fontenc}
\usepackage{wrapfig}
\usepackage{todonotes}

\usepackage{stmaryrd} 
\usepackage{mathtools} 
\usepackage[normalem]{ulem} 

\usepackage[
            pdfauthor   = {Markus Lange-Hegermann},
            pdftitle    = {},
            pdfsubject  = {},
            pdfkeywords = {},
            bookmarks=true,
            bookmarksopen=true,
            colorlinks=true,
            hyperindex=true,
            linkcolor=blue,
            citecolor=blue,
            urlcolor=blue,
            ]{hyperref}

\newcommand{\mcI}{{\mathcal{I}}}

\newcommand{\Z}{{\mathbb{Z}}}
\newcommand{\Q}{{\mathbb{Q}}}

\newcommand{\C}{{\mathbb{C}}}
\DeclareMathOperator{\ld}{ld}
\DeclareMathOperator{\ini}{ini}
\DeclareMathOperator{\ord}{ord}
\newcommand{\sol}{\ensuremath{\mathfrak{Sol}}}
\newcommand{\forget}{\ensuremath{\rho}}

\newcommand{\ce}{\zeta}
\newcommand{\ma}{\ensuremath{\mathbf{a}}}
\newcommand{\cd}{\bar{c}}

\newtheorem{theorem}{Theorem}[section]
\newtheorem{lemma}[theorem]{Lemma}
\newtheorem{proposition}[theorem]{Proposition}

\theoremstyle{definition}
\newtheorem{definition}[theorem]{Definition}
\newtheorem{example}[theorem]{Example}
\newtheorem{remark}[theorem]{Remark}

\author{Markus Lange-Hegermann}
\address{Lehrstuhl B f\"ur Mathematik, Rheinisch-Westf\"alische Technische Hochschule Aachen, 52062 Germany}
\email{\href{mailto:Markus Lange-Hegermann <markus.lange.hegermann@rwth-aachen.de>}{markus.lange.hegermann@rwth-aachen.de}}

\begin{document}

\title[Differential Counting Polynomial]{The Differential Counting Polynomial}

\begin{abstract}
  The aim of this paper is a quantitative analysis of the solution set of a system of polynomial nonlinear differential equations, both in the ordinary and partial case.
  Therefore, we introduce the differential counting polynomial, a common generalization of the dimension polynomial and the (algebraic) counting polynomial.
  Under mild additional assumptions, the differential counting polynomial decides whether a given set of solutions of a system of differential equations is the complete set of solutions.
\end{abstract}


\maketitle


\section{Introduction}

Many systems of differential equations do not admit closed form solutions in ``elementary'' functions and hence cannot be solved symbolically.
Despite this, increasingly good heuristics are implemented in computer algebra systems to find solutions \cite{ChebTerrab1995102,ChebTerrabRoche2008}.
Given such a set of closed form solutions returned by a computer algebra system, the question remains whether this set is the complete solution set.
More generally, the goal of this paper is to ``measure'' the sizes of solution sets $U\subseteq V$ in order to decide whether $U=V.$

There are many classical measures of the size of the solution set (cf.\ \cite{Seiler} for an overview), the strongest\footnote{It implies other descriptions like Cartan characters or Einstein's strength.} of which is Kolchin's dimension polynomial \cite{KolPolynomial,DiffDimPoly}.
However, the dimension polynomial can only describe solution sets given by characterizable differential ideals and only the dimension of such solution sets, but no finer details (cf.\ Example~\ref{example_better_dim_poly}).
It was a great surprise to the author that such finer details can appear in the solution set of a differential equation, for example countable infinite exceptional sets (cf.\ Example~\ref{example_hard}).

This paper introduces the differential counting polynomial, a more detailed description of a solution set of a system of differential equations.
If it exists, it generalizes the dimension polynomial (Theorem~\ref{theorem_counting_and_dimension}) and decides in many cases whether solution sets are equal (Theorem~\ref{theorem_differential_counting_subsets} and Proposition~\ref{proposition_counting_poly_estimate}).

The idea of the differential counting polynomial originates from Ples\-ken's algebraic counting polynomial \cite{PleskenCounting}.
The algebraic counting polynomial is an element $c(V)\in\Z[\infty]$ which describes the size of a constructible set $V$ in affine $n$-space.
Here, $\infty$ is a free indeterminate, which can be thought of as representing the cardinality of an affine $1$-space.
For example, the algebraic counting polynomial of an affine $i$-space is $\infty^i\in\Z[\infty]$, and if $V$ is a $j$-fold unramified cover of $W$, then $c(V)=j\cdot c(W)\in\Z[\infty]$.
Two constructible sets $U\subseteq V$ are equal if and only if their algebraic counting polynomials $c(U)$ and $c(V)$ coincide.

The differential counting polynomial is the algebraic counting polynomial for the different Taylor polynomials of degree $\ell$ of all formal power series solutions.
We restrict to formal power series solutions as they exist in a formally consistent system of differential equations for any formal power series given as initial data.
Similar results hold for analytic \cite{Riquier} but not for smooth initial data (cf.\ Lewy's example \cite{LewyConterexample}).

Determining the differential counting polynomial of a set of differential equations is not algorithmic in general.
Even in the case of a single inhomogeneous linear differential equation the problem of the existence of formal power series solutions can be reduced to Hilbert's unsolvable tenth problem about Diophantine equations \cite{DenefLipshitz}.
Also, the existence of the differential counting polynomial is still an open problem.
However, the author succeeded in computing the differential counting polynomial using Theorem~\ref{theorem_without_inequations} and inductive proofs similar to Examples~\ref{example_hard} and \ref{example_better_dim_poly} for all of the various classes of examples of differential equations he encountered.

So, it is hard to determine the differential counting polynomial from a set of differential equations without \emph{explicitly} knowing the corresponding full set of solutions $V.$
In contrast, the differential counting polynomial $c(U)$ of an explicitly given set of solutions $U$ can be computed more easily by determining how unrestrictedly the power series coefficients of elements in $U$ are chooseable (cf.\ Example~\ref{example_better_dim_poly}).
Once the differential counting polynomial $c(V)$ of $V$ is known, one can often decide whether an explicitly given set of solutions $U$ is equal to the complete set of solutions $V$ by comparing $c(U)$ and $c(V)$.

Sections~\ref{sect_simple_systems} and \ref{sect_alg_counting} recapitulate simple systems and Plesken's algebraic counting polynomial, respectively, and generalize them for our needs.
In Section~\ref{section_differential_counting} we define the differential counting polynomial, state its basic properties, and give examples.
The author's PhD thesis \cite{lhphd} contains additional (classes of) examples.
The proofs follow in Section~\ref{sect_proofs}.

\section{Simple \texorpdfstring{$\sigma$}{sigma}-Systems}\label{sect_simple_systems}

Simple systems stratify constructible sets into sets with convenient geometric properties.
Such systems underlie Plesken's counting polynomial.
This section recapitulates simple systems and generalizes them to describe differential equations.

Let $R:=\C[y_1,\ldots,y_n]$ be a polynomial ring.
We fix the total order, called ranking, $y_1<y_2<\ldots<y_n$ on $\{y_1,\ldots,y_n\}$.
The $<$-greatest variable $\ld(p)$ occurring in $p\in R\setminus \C$ is called leader of $p$.
The coefficient $\ini(p)$ of the highest power of $\ld(p)$ in $p$ is called the initial of $p$.
For $S\subset R\setminus\C$ define $\ld(S):=\{\ld(p)|p\in S\}$ and similarly $\ini(S)$.
Denote by $S_{<y_i}$, $S_{\le y_i}$, and $S_{y_i}$ the sets $S\cap \C[y_1,\ldots,y_{i-1}]$, $S\cap \C[y_1,\ldots,y_i]$, and $\{p\in S\mid\ld(p)=y_i\}$, respectively, for all $1\le i\le n$.

We call a set of finitely many equations and countably many inequations a \textbf{$\sigma$-system}.
If this set is finite, we call it \textbf{system}.
Let $S$ be a $\sigma$-system over $R$.
We denote the set of solutions in $\C^n$ of $S$ by $\sol(S)$.
Call $S$ weakly triangular if it contains no equation or inequation in $\C$ and it contains either at most one equation or arbitrary many inequations of leader $y_i$ for each $1\le i\le n$.
We say that $S$ has non-vanishing initials if no initial vanishes when substituting an $\ma \in \sol(S)$.
Substituting all indeterminates $y_i\not=\ld(p)$ in $p\in S$ by an $a_i\in \C$ results in a univariate polynomial.
If all these univariate polynomials resulting from the $p\in S$ and $\ma \in \sol(S)$ are square-free, then we call $S$ square-free.
We call $S$ a \textbf{simple $\sigma$-system} if it is weakly triangular, has non-vanishing initials, and is square-free.

A set $\{S_1\ldots,S_l\}$ of simple $\sigma$-systems with disjoint solution sets is called an \textbf{algebraic Thomas decomposition} of a $\sigma$-system $S$ if $\sol(S)=\biguplus_{1\le i\le l}\sol(S_i)$.
Such a Thomas decomposition is called \textbf{comprehensive} with respect to an indeterminate $y_k$ if $\sol((S_i)_{\le y_k})\cap\sol((S_j)_{\le y_k})\in\{\emptyset,\sol(S_i)\}$ for all $1\le i,j\le l$.

\section{Algebraic Counting Polynomials}\label{sect_alg_counting}
\label{sect_algebraic_thomas_counting}

This section recapitulates Plesken's algebraic counting polynomial for constructible sets \cite{PleskenCounting,BaechlerPlesken}, and generalizes it to be suitable for describing differential equations.
We consider the affine $n$-space $\C^n$ with projections $\pi_i:\C^n\to\C^i:(a_1,\ldots,a_n)\mapsto(a_1,\ldots,a_i)$.

\begin{definition}\label{definition_counting}
  The following four axioms iteratively\footnote{This is independent of the order in which the axioms are applied \cite[Prop.~3.3]{PleskenCounting}.} applied to a constructible set $V\subseteq\C^n$ yield its \textbf{algebraic counting polynomial}, an element in the univariate polynomial ring $\Z[\infty]$.
  \begin{enumerate}
    \item $c(V)=|V|$ if $V$ is finite.\label{definition_counting_singleton}
    \item $c(V)=\infty$ for an affine $1$-space $V$ over $\C$.\label{definition_counting_1_dim}
    \item $c(V\uplus W)=c(V)+c(W)$ for disjoint constructible sets $V,W\subset\C^n$.\label{definition_counting_sum}
    \item If $V\subset\C^n$ is constructible and for some $1\le i\le n$ each non-empty fiber $W$ of $\pi_i$ has the same value under $c$, then $c(V)=c(W)\cdot c(\pi_i(V))$.\label{definition_counting_product}
  \end{enumerate}
\end{definition}

The algebraic Thomas decomposition makes the computation of the algebraic counting polynomial algorithmic  \cite{PleskenCounting}.
The following theorem shows how the algebraic counting polynomial can be used to compare constructible sets.
Our goal is a similar theorem for solution sets of differential equations.

\begin{theorem}[{\cite[Cor.~3.4]{PleskenCounting}}]\label{theorem_counting_polynomials}
  Let $U\subseteq V\subseteq\C^n$ be constructible sets.
  Then $U=V$ if and only if $c(U)=c(V)$.
\end{theorem}

In solution sets of differential equations, countable exceptional sets appear naturally (cf.\ Example~\ref{example_hard}).
To describe these sets, we generalize the algebraic counting polynomial.

\begin{definition}
  Let $V\subseteq \C^n$.
  Then, call \emph{any} element $c(V)$ in the polynomial ring $\Z[\infty,\aleph_0]$ constructed by iteratively applying the four axioms from Definition~\ref{definition_counting} above and the following fifth axiom an \textbf{algebraic counting polynomial} of $V$.
  \begin{enumerate}
    \item[(5)] $c(\C^1\setminus M)=\infty-\aleph_0$ for $M\subset\C^1$ is countably infinite.\label{definition_counting_discrete}
  \end{enumerate}
\end{definition}

\begin{remark}\label{remark_counting_not_unique}
  In general, the algebraic counting polynomial is not unique.
  For example, the set $\sol(\{x-i\not=0|i\in\Z_{\ge0}\})=\sol(\{x-i\not=0|i\in\Z_{\ge1}\})\uplus\{0\}$ can have both algebraic counting polynomial $\infty-\aleph_0$ and $\infty-\aleph_0+1$.
  Hence, Theorem~\ref{theorem_counting_polynomials}, which states that the algebraic counting polynomial decides equality of contained \emph{constructible} sets, cannot hold in general, but it holds for the important special case of well-fibered sets (cf.\ Theorem~\ref{theorem_counting_subset}).
\end{remark}

Even worse, it is not clear in which cases an algebraic counting polynomial exists, i.e.\ that there exists a way to apply the axioms terminating in an element in $\Z[\infty,\aleph_0]$.
Simple algebraic $\sigma$-systems are a first example where existence (and some uniqueness) holds.

\begin{theorem}\label{theorem_counting_well_defined}
  Let $S\subset \C[y_1,\ldots,y_n]$ be a simple algebraic $\sigma$-system.
  Then an algebraic counting polynomial of $\sol(S)$ exists.
  
  Consider counting polynomials as polynomials in the indeterminate $\infty$.
  Then degree and leading coefficient of any algebraic counting polynomial $c(\sol(S))$ are equal to those of the (unique, cf.\ \cite[Prop.~3.3]{PleskenCounting}) counting polynomial $c(\overline{\sol(S)})$ of its Zariski closure $\overline{\sol(S)}$.
  In particular, the degree of $c(\sol(S))$ is equal to the dimension of $\overline{\sol(S)}$.
\end{theorem}

In particular, the degrees and leading coefficients of these algebraic counting polynomials are well defined and the leading coefficients are natural numbers.
We postpone the proof of this theorem to page~\pageref{proof_theorem_counting_well_defined}.

Countable exceptional sets complicate the use of the algebraic counting polynomial for $\sigma$-systems in applications.
However, the counting polynomial is well behaved for \textbf{well-fibered} sets, which we define as sets with a counting polynomial in $\Z[\infty]$.
These sets behave similarly as constructible sets with regard to algebraic counting polynomials; in particular the algebraic counting polynomial is strong enough to decide equality of such sets contained in each other.

\begin{theorem}\label{theorem_counting_subset}
  Let $U\subseteq V\subseteq\C^n$ be two well-fibered sets.
  Then $U=V$ if and only if $c(U)=c(V)$.
  In particular, the counting polynomial of well-fibered sets is well-defined in the sense that it is unique.
\end{theorem}

We postpone the proof of this theorem to page~\pageref{subsection_proof_theorem_counting_subset}.

Even when $\aleph_0$ appears in algebraic counting polynomials of two sets $U\subseteq V$, one might be able to prove $U\not=V$ by estimating the algebraic counting polynomial.
First, any subset of $\C^1$ with a countably infinite complement can be enlarged to a set with finite complement.
Second, any subset of $\C^1$ with a countably infinite complement can be shrunk to a finite set.
Thus, for $p(\aleph_0,\infty)\in\Z[\aleph_0,\infty]$ and $q(\infty)\in\Z[\infty]$ we define $p\prec q$ if $p(\infty-k,\infty)=q(\infty)$ and $p\succ q$ if $p(k,\infty)=q(\infty)$ for some $k\in\Z_{\ge0}$.
Additionally, we use the total order $q<q'$ if there exists an $x_0$ with $q(x)<q'(x)$ for all $x>x_0$ for $q,q'\in\Q[\infty]$.

\begin{proposition}\label{proposition_estimate_counting_polynomial}
  Let $U\subseteq V\subseteq\C^n$ have algebraic counting polynomials $p_1(\aleph_0,\infty):=c(U)$ and $p_2(\aleph_0,\infty):=c(V)$.
  If there exist $q_1,q_2\in\Z[\infty]$ with $p_1\prec q_1\lneqq q_2\prec p_2$, then $U\not=V$.
\end{proposition}

The proof of this proposition is a natural generalization of one implication of the proof of \cite[Cor.~3.4]{PleskenCounting}.

\section{The Differential Counting Polynomial}\label{section_differential_counting}

This section defines the differential counting polynomial and states some of its properties.
Beforehand, we fix some notation.

\subsection{Preliminaries}

Let $F\supseteq \C$ be a field of meromorphic functions in $n$ complex variables $x_1,\ldots,x_n$, and $\Delta=\{\partial_{x_1},\ldots,\partial_{x_n}\}$ the corresponding set of partial differential operators.
Let $U:=\{ u^{(1)}, \ldots, u^{(m)} \}$ be a set of differential indeterminates and define $u^{(j)}_\mu:=\partial^\mu u^{(j)}$ for $\partial^\mu := \partial^{\mu_1}_{x_1}\ldots\partial^{\mu_n}_{x_n}$, $\mu \in (\Z_{\ge 0})^n$.
The differential polynomial ring $F\{U\}$ is the infinitely generated polynomial ring in the indeterminates $\{U\}_\Delta:=\{u^{(j)}_\mu | 1\le j\le m, \mu \in (\Z_{\ge 0})^n\}$.
Denote by $F\{U\}_{\le \ell}$ its subring of elements of order at most $\ell$
The derivations $\partial_{x_i}: F \to F$ extend to $\partial_{x_i}: F\{U\} \to F\{U\}$ via additivity and Leibniz rule.
Let $\operatorname{sep}(p)$ be the separant of $p\in F\{U\}$.
A ranking of $F\{U\}$ is a total ordering $<$ of $\{U\}_\Delta$ satisfying the two properties (1) $u^{(j)}_\mu<\partial u^{(j)}_\mu$ and (2) $u^{(j)}_\mu<u^{(j')}_{\mu'}$ implies $\partial u^{(j)}_\mu<\partial u^{(j')}_{\mu'}$ for all $u^{(j)}_\mu, u^{(j')}_{\mu'}\in\{U\}_\Delta$ and $\partial\in\Delta$.
A ranking $<$ is called orderly if $|\mu|<|\mu'|$ implies $u^{(j)}_{\mu} < u^{(j')}_{\mu'}$, where $|\mu|:=\mu_1+\ldots+\mu_n$.
In what follows, we fix an orderly ranking $<$ on $F\{U\}$.

Now, we extend the formalism of differential algebra to incorporate algebraic constraints for power series coefficients.
We consider the set 
\begin{align*}
  G:=G(U,\Delta):=\left\{g_{\mu}^{(j)}\mid \mu\in\Z_{\ge0}^n,1\le j\le m\right\}
\end{align*}
of indeterminates and call the polynomial ring $\C[G]$ the \textbf{polynomial ring of power series coefficients}.
The bijection $\forget:\{U\}_\Delta\to G(U,\Delta):u^{(j)}_{\mu}\mapsto g^{(j)}_{\mu}$ extends the orderly ranking $<$ on $F\{U\}$ to an (algebraic) ranking on $\C[G]$.
For $\ell\in\Z_{\ge0}$ let $\C[G]_{\le\ell}$ be the subring generated by all indeterminates $g_{\mu}^{(j)}$ of order $|\mu|\le\ell$.

We call a union of finitely many differential equations in $F\{U\}$, finitely many power series coefficient equations in $\C[G]$, and countably many power series coefficient inequations in $\C[G]$ an \textbf{algebraically restricted $\sigma$-system of differential equations} and an \textbf{algebraically restricted system of differential equations} if it is finite.

We are concerned with power series solutions in the power series ring $P_\ce:=\C[[x_1-\ce_1,\ldots,x_n-\ce_n]]$ centered around $\ce=(\ce_1,\ldots,\ce_n)\in\C^n$.
We interpret equations in $F\{U\}$ as equations for functions, e.g., $u^{(1)}_{(0,\ldots,0)}=0$ implies that $u^{(1)}$ is the zero function, whereas equations in $\C[G]$ are equations for single power series coefficients, e.g., $g^{(1)}_{(0,\ldots,0)}=0$ implies that $u^{(1)}$ has a zero at its center of expansion $\ce$.
More precisely, a solution of power series coefficient equations or inequations is defined as a tuple of power series $f\in P_\ce^U\cong\bigoplus_UP_\ce$ that evaluates to zero or non-zero, respectively, when substituting $g_{\mu}^{(j)}$ by the coefficient of $(x_1-\ce_1)^{\mu_1}\ldots(x_n-\ce_n)^{\mu_n}$ in $f(u^{(j)})$.
Our definition of a solution of a differential equation in $F\{U\}$, where all coefficients are holomorphic in $\ce$, is the usual one.
Denote the set of formal power series solutions of an algebraically restricted $\sigma$-system of differential equations $S$ around $\ce$ by $\sol_\ce(S)\subseteq P_\ce^U$.

We consider Taylor polynomials which extend to a Taylor series.
Let $P_{\ce,>\ell}^U$ be the $P_\ce$-sub\-mo\-dule of $P_\ce^U$ generated by the $u^{(j)}\mapsto (x_1-\ce_1)^{\mu_1}\ldots(x_n-\ce_n)^{\mu_n}$ for $\mu\in\Z_{\ge0}^n$ with $|\mu|>\ell$.
Define the set $\sol_\ce(S)_{\le \ell}$ of \textbf{formal power series solutions of $S$ around $\ce$ truncated at order $\ell$} as the image of $\sol_\ce(S)$ in $P_\ce^U/P_{\ce,>\ell}^U$ under the natural epimorphism $P_\ce^U\twoheadrightarrow P_\ce^U/P_{\ce,>\ell}^U$.

\subsection{Definition of the Differential Counting Polynomial}

The algebraic Thomas decomposition computes the algebraic counting polynomial.
For differential equations, there is a similar decomposition.

\begin{theorem}\label{theorem_sigma_describe_solutions}
  Let $S$ be an algebraically restricted system of differential equations, such that the center of expansion $\ce\in\C^n$ is not a pole of any coefficient of a differential equation.
  Let $\ell\in\Z_{\ge0}$.
  There exists a countable set $C$ of simple algebraic $\sigma$-systems in $\C[G]_{\le \ell}$ with
  \begin{align*}
    \sol_\ce(S)_{\le \ell}=\biguplus_{\widetilde{S}\in C}\sol_\ce(\widetilde{S})_{\le\ell}\mbox{ .}
  \end{align*}
\end{theorem}

We postpone the proof of this theorem to page \pageref{theorem_sigma_describe_solutions_proof}.
This theorem justifies the following definition of the differential counting polynomial.

\begin{definition}
  \label{def_counting}
  Let $S$ be an algebraically restricted system of differential equations.
  Let $C_\ell$ be a countable set of algebraic $\sigma$-systems with $\sol_\ce(S)_{\le \ell}=\biguplus_{\widetilde{S}\in C_\ell}\sol_\ce(\widetilde{S})_{\le\ell}$ for each $\ell\in\Z_{\ge0}$.
  
  If an algebraic counting polynomial exists for $C_\ell$, then define an \textbf{$\ell$-th differential counting polynomial} of $S$ as $c(C_\ell)\in\Z[\infty,\aleph_0]$.
  If an $\ell$-th differential counting polynomial exists for all $\ell$, then define a \textbf{counting sequence} $c(S)\in \Z[\infty,\aleph_0]^{\Z_{\ge 0}}$ of $S$ (or $\sol_\ce(S)$) as 
          \begin{align*}
            c(S): \ell\mapsto c(C_\ell)\mbox{ .}
          \end{align*}
  If there exists a $p\in\Q[\ell,\aleph_0,\infty,\infty^\ell,\infty^{\frac{\ell^2}{2!}},\ldots,\infty^{\frac{\ell^n}{n!}}]$ such that $c(S)(\ell)=p$ for ultimately all $\ell$, then call $p$ a \textbf{differential counting polynomial} of $S$ (or $\sol_\ce(S)$) and denote it by $\cd(S)$.
  For a differential ideal $I=\langle p_1,\ldots,p_k\rangle_\Delta$ define $c(I):=c(\{p_1,\ldots,p_k\})$ and $\cd(I):=\cd(\{p_1,\ldots,p_k\})$.
  
  We write $\infty^{\ell^2}$ instead of $(\infty^{\frac{\ell^2}{2!}})^2$ and use similar simplifications.
\end{definition}

The existence of a differential counting sequence or a differential counting polynomial is not clear, in general.

\subsection{Deciding Equality of Sets}

The following theorem and proposition use counting sequences and differential counting polynomials to decide equality of sets contained in each other.

\begin{theorem}\label{theorem_differential_counting_subsets}
  Let $S_1,S_2$ be two algebraically restricted systems of differential equations with $\sol_\ce(S_1)\subseteq\sol_\ce(S_2)$.
  \begin{enumerate}
  \item Assume that both counting sequences $c(S_1)$ and $c(S_2)$ exist and that $c(S_1)(\ell), c(S_2)(\ell)\in\Z[\infty]$ for all $\ell\in\Z_{\ge0}$.\label{theorem_differential_counting_subsets_1}
        Then $\sol_\ce(S_1)=\sol_\ce(S_2)$ if and only if $c(S_1)(\ell)=c(S_2)(\ell)$ for all $\ell\in\Z_{\ge0}$.
        In particular, $c(S_1)$ is the unique counting sequence of $S_1$.
  \item Assume both differential counting polynomials $\cd(S_1)$ and $\cd(S_2)$ exist and $\cd(S_1),\cd(S_2)\in\Q[\ell,\infty,\infty^\ell,\infty^{\frac{\ell^2}{2!}},\ldots,\infty^{\frac{\ell^n}{n!}}]$ holds.\label{theorem_differential_counting_subsets_2}
        Then $\sol_\ce(S_1)=\sol_\ce(S_2)$ if and only if $\cd(S_1)=\cd(S_2)$.
        In particular, $\cd(S_1)$ is the unique differential counting polynomial of $S_1$.
  \end{enumerate}
\end{theorem}

The sets $\sol_\ce(S_1)_{\le \ell}$ and $\sol_\ce(S_2)_{\le \ell}$ of formal power series solutions truncated at order $\ell$ are well-fibered under the conditions of \eqref{theorem_differential_counting_subsets_1} and well-fibered for high enough $\ell$ under the conditions of \eqref{theorem_differential_counting_subsets_2}.
Thus, this theorem is a corollary of Theorem~\ref{theorem_counting_subset}.

Remark~\ref{remark_counting_not_unique} indicates that a stronger version of this theorem is unlikely.
However, we can show that two sets are not equal in the differential case similar to Proposition~\ref{proposition_estimate_counting_polynomial}, by using the total order $<$ and the estimation $\prec$, both defined before Proposition~\ref{proposition_estimate_counting_polynomial}.

\begin{proposition}\label{proposition_counting_poly_estimate}
  Let $S_1,S_2$ be two algebraically restricted systems of differential equations with $\sol_\ce(S_1)\subseteq\sol_\ce(S_2)$ such that the counting sequences $c(S_1)$ and $c(S_2)$ exist.
  If there exist an $\ell\in\Z_{\ge0}$ and $q_1,q_2\in\Z[\infty]$ with $c(S_1)(\ell)\prec q_1\lneqq q_2\prec c(S_2)(\ell)$, then $\sol_\ce(S_1)\not=\sol_\ce(S_2)$.
\end{proposition}

This proposition follows from Proposition~\ref{proposition_estimate_counting_polynomial} just as Theorem~\ref{theorem_differential_counting_subsets} follows from Theorem~\ref{theorem_counting_subset}.

\subsection{Comparison to the Differential Dimension Polynomial}

The counting sequence and the differential counting polynomial are connected to the differential dimension polynomial (cf.\ \cite{KolPolynomial,DiffDimPoly,DifferentialDimensionPolynomials,JohnsonPolynomial}) in the version defined for characterizable differential ideals (cf.\ \cite{BoundRosenfeldGroebner,Hubert2000}).

For the following theorem, we consider an $\ell$-th differential counting polynomial as a polynomial in the indeterminate $\infty$ and coefficients in $\Z[\aleph_0]$.
Similarly, we consider a differential counting polynomial as a polynomial in the indeterminates $\infty^{\frac{\ell^i}{i!}}$ for $0\le i\le n$ and coefficients in $\Q[\aleph_0,\ell]$.
We order the indeterminates $\infty^{\frac{\ell^n}{n!}}>\ldots>\infty^{\frac{\ell^2}{2!}}>\infty^\ell>\infty$.

\begin{theorem}\label{theorem_counting_and_dimension}
  Let $I:=\langle S\rangle_\Delta:(\ini(S)\cup\operatorname{sep}(S))^\infty$ be a characterizable differential ideal given by a regular chain $S$.
  Denote by
  \begin{align*}
    \Omega_I:\Z_{\ge0}\mapsto\Z_{\ge0}:\ell\mapsto\dim(F\{U\}_{\le \ell}/(I\cap F\{U\}_{\le \ell}))
  \end{align*}
  its differential dimension function and by $\omega_I(l)$ its dimension polynomial, the unique polynomial that agrees with the dimension function for all large enough $\ell$.
  If an $\ell$-th differential counting polynomial $c(I)(\ell)$ of $I$ exists, then its leading term is
  \begin{align*}
    \big(\prod_{\mathclap{\substack{p\in S\\\ord(p)\le \ell}}}\deg_{\ld(p)}p\big) \cdot \infty^{\Omega_I(\ell)}\mbox{ .}
  \intertext{%
  If a differential counting polynomial $\cd(I)$ of $I$ exits, then its leading term is
  }
    \big(\prod_{p\in S}\deg_{\ld(p)}p\big) \cdot \infty^{\omega_I(\ell)}\mbox{ .}
  \end{align*}
\end{theorem}

We postpone the proof to page~\pageref{proof_theorem_counting_and_dimension}.

Under the assumptions of this theorem, the differential counting polynomial implies the same invariants of differential birational maps as the differential dimension polynomial.
In particular, the differential type, typical dimension, and differential dimension can be read off the exponent of $\infty$ in the leading term.
This exponent is a polynomial in $\ell$ equal to the differential dimension polynomial.
The differential type $t$ is the degree of this exponent, and when writing it as $\sum_{i=0}^na_i\binom{\ell+i}{i}$ the typical dimension is $a_t$ and the differential dimension is $a_n$ (cf.\ \cite[Theorem~1.1]{DiffDimPoly}).

\subsection{Simple Differential Systems without Inequations}

For semilinear systems of differential equations there exists a closed formula for the differential counting polynomial that holds once all differential consequences are obvious from the system (such systems are called passive \cite{Janet}, involutive \cite{thomasalg_jsc} or coherent \cite{Kol}).
It follows from this formula, that the differential counting polynomial of such systems does not involve $\aleph_0$.
This holds for the more general class of differential equations given by a simple differential system $S$ without inequations \cite[Def.~3.5]{thomasalg_jsc}.
Let $\mcI(S):=\langle T\rangle_\Delta:q^\infty$ be the corresponding characterizable differential ideal, where $T$ are the equations in $S$ and $q$ is the product of the initials and separants of $T$.
Let $\Omega_{\mcI(S)}$ denote its differential dimension function and $\omega_{\mcI(S)}$ its dimension polynomial \cite{DiffDimPoly}.

\begin{theorem}\label{theorem_without_inequations}
  Let $S=\{p_1,\ldots,p_s\}$ be a simple differential system in $F\{U\}$ without inequations.
  Consider formal power series solutions around a point $\ce\in\C^n$ such that neither evaluating the coefficients of $S$ at $\ce$ yields a pole nor any initial or separant vanishes identically.  
  Then, its unique counting sequence is
  \begin{align*}
    c(S):l\mapsto \Big(\prod_{{\substack{1\le i\le s\\ \ord(p_i)\le \ell}}}\deg_{\ld(p_i)}(p_i)\Big)\cdot\infty^{\Omega_{\mcI(S)}(\ell)}\mbox{ ,}
  \end{align*}
  and its differential counting polynomial is
  \begin{align*}
    \cd(S)=\Big(\prod_{{1\le i\le s}}\deg_{\ld(p_i)}(p_i)\Big)\cdot\infty^{\omega_{\mcI(S)}(\ell)}\mbox{ .}
  \end{align*}
\end{theorem}
We postpone the proof to page~\pageref{proof_theorem_without_inequations}.

Differential inequations in the sense of Thomas (cf.\ \cite{thomasalg_jsc}) are not well-suited to count power series solutions, as $u(x)\not=0$ just implies that at least one power series coefficient of $u$ is non-zero.

Many examples of systems of differential equations yield a Thomas decomposition into a single simple differential system without inequations.
Examples are systems of linear differential equations and semilinear formally integrable systems of differential equations.
We show an example of the latter class.

\begin{example}\label{example_navier_stokes_counting}
  Let $F=\C$, $\Delta=\{\partial_x,\partial_y, \partial_z, \partial_t \}$,  $U=\{u,v,w,p\}$, and fix a ranking, such that the leaders are the underlined indeterminates.
  The incompressible Navier-Stokes equations are
  \begin{align*}
    S:= \{ && u_t+uu_x+vu_y+wu_z+p_x-\left(\underline{u_{xx}}+u_{yy}+u_{zz}\right) &= 0,\\
           && v_t+uv_x+vv_y+wv_z+p_y-\left(\underline{v_{xx}}+v_{yy}+v_{zz}\right) &= 0,\\
           && w_t+uw_x+vw_y+ww_z+p_z-\left(\underline{w_{xx}}+w_{yy}+w_{zz}\right) &= 0,\\
           && \underline{u_x}+v_y+w_z &= 0 \quad \}.
  \end{align*}
  A differential Thomas decomposition for $S$ is given by the one system
  \begin{align*}S\cup\left\{\ 2u_yv_x+2u_zw_x+2v_zw_y+u_x^2+v_y^2+w_z^2+\underline{p_{xx}}+p_{yy}+p_{zz}=0 \right\}\mbox{ ,}\end{align*}
  where the Poisson pressure equation is added to $S$.
  In particular, the Thomas decomposition of $S$ does not contain any inequation.
  A combinatorial calculation shows that the differential counting polynomial of the incompressible Navier-Stokes equations is $\infty^{\ell^3+\frac{11}{2}\ell^2+\frac{17}{2}\ell+4}$.
\end{example}

\subsection{Examples}

To transform a differential equation into an equation for a single power series coefficient, we define the partial map $\forget:F\{U\}\to\C[G]$ as additive, multiplicative, mapping $u^{(j)}_{\mu}$ to $g^{(j)}_{\mu}$, and mapping any meromorphic $f\in F$ to $f(\ce)\in\C$ if it has no pole in $\ce$.\label{forget}

We call the following the \textbf{postponing} of a differential equation $p\in F\{U\}$:
Replace $p$ by its first derivatives $\{\partial_{x_1}p,\ldots,\partial_{x_n}p\}$ and $\forget(p)$; this does not change the solution set.

In the following example, there is a power series coefficient that can be chosen arbitrarily except for a countable infinite exceptional set for a solution to exist.
This exceptional set corresponds to the indeterminate $\aleph_0$ in the differential counting polynomial.
In particular, there exists a set of differential equations for which the indeterminate $\aleph_0$ appears in the differential counting polynomial.
This countable exceptional set carries over from the set of formal power series solutions to the set analytical solutions, as all formal power series solutions in this example have a positive radius of convergence.
The author did not find a similar set of differential equations which describes natural phenomena or appears in the scientific literature.

\begin{example}\label{example_hard}
  Let $U=\{u^{(1)},u^{(2)}\}$, $\Delta=\{\partial_t\}$, $F=\C(t)$, and $<$ the orderly ranking with $u^{(1)}>u^{(2)}$.
  We show the following.
  For all $\ell\ge1$
  \begin{align*}
    \cd(S)=c(S)(\ell)=\infty^3-\infty^2+\infty-\aleph_0
  \end{align*}
  for formal power series solutions of $S:=\{p:=u^{(2)}u^{(1)}_1-u^{(1)}+\frac{1}{t}=0, u^{(2)}_2=0\}$ centered around $\ce\in\C\setminus\{0\}$.
  Each of these solutions is locally convergent and $S$ has no solutions centered around $0$.

  Use the ansatz $u^{(1)}(t)=\sum_{i=0}^\infty g^{(1)}_i\frac{(t-\ce)^i}{i!}$ and $u^{(2)}(t)=\sum_{i=0}^\infty g^{(2)}_i\frac{(t-\ce)^i}{i!}$.
  Adding $g^{(2)}_0\not=0$ to $S$ yields $T:=\{p = 0, u^{(2)}_2 = 0, g^{(2)}_0 \not=0 \}$.
  It has $\ell$-th differential counting polynomial $c(T)(\ell)=\infty^3-\infty^2$ for every order $\ell\ge1$.
  This follows by means of the proof of Theorem~\ref{theorem_without_inequations} on page \pageref{proof_theorem_without_inequations}; the inequation $g^{(2)}_0 \not=0$ ensures that the initials of the derivatives of $p$ are non-zero after applying $\forget$.

  The system $S\cup \{g^{(2)}_0=0\}$, which is complementary to the previously treated system $S\cup \{g^{(2)}_0\not=0\}$, is equivalent to
  \begin{align*}
  S_1:=\{ && \partial_tp = u^{(2)}u^{(1)}_2+(u^{(2)}_1-1)u^{(1)}_1-\frac{1}{t^2}   &= 0, & u^{(2)}_2 &=0, \\
          && g^{(1)}_0-\frac{1}{\ce} &=0, & g^{(2)}_0 &=0 \} \mbox{ ,}
  \end{align*}
  by postponing $p$.
  This system $S_1$ belongs to the family 
  \begin{align*}
  S_k:=\{ q_k:=u^{(2)}u^{(1)}_{k+1}+(ku^{(2)}_1-1)u^{(1)}_k+(-1)^{k}\frac{k!}{t^{k+1}} &= 0, && u^{(2)}_2 = 0,\\
          (ig^{(2)}_1-1)g^{(1)}_i+(-1)^i\frac{i!}{\ce^{i+1}} &= 0                  && \forall\ 0\le i<k,\\
          {\textstyle\prod_{i=1}^{k-1}(ig^{(2)}_1-1)} & \not=0,            && g^{(2)}_0 =0 \}
  \end{align*}
  of systems.
  Here $q_k$ results from differential reduction of $\partial_t^kp$ by $u^{(2)}_2$. 
  After application of $\forget$ and reduction with elements in $S_k$, the differential equations $q_k$ yield $(kg^{(2)}_1-1)g^{(1)}_k+(-1)^{k}\frac{k!}{\ce^{k+1}}$.
  To ensure a non-zero initial, we add $(kg^{(2)}_1-1)\not=0$ to $S_k$.
  Then, postponing $q_k$, which after reduction by $u^{(2)}_2$ results in $q_{k+1}$, yields the system $S_{k+1}$.
  Complementary, when adding $kg^{(2)}_1-1=0$ to $S_k$, the system is inconsistent, since reducing $\partial_tq_k$ by $u^{(2)}_2$ results in $q_{k+1}$.
  Then
  \begin{align*}
    \forget(q_{k+1})=g^{(2)}_0g^{(1)}_{k+2}+(kg^{(2)}_1-1)g^{(1)}_{k+1}+(-1)^{k+1}\frac{(k+1)!}{\ce^{k+2}}=0
  \end{align*}
  yields the contradiction $(-1)^{k+1}\frac{(k+1)!}{\ce^{k+2}}=0$ by using the relations $g^{(2)}_0=0$ and $kg^{(2)}_1-1=0$.

  Study the remaining system $S_\infty:=\bigcup_{i=1}^\infty S_i$.
  The equations $(kg^{(2)}_1-1)g^{(1)}_k+(-1)^k\frac{k!}{\ce^{k+1}}=0$ make $p$ superfluous.
  Furthermore, $g^{(2)}_1$ cannot equal $\frac{1}{k}$ for any $k\in\Z_{\ge1}$.
  Thus, there exists a countable infinite set of exceptional values for the power series coefficient $g^{(2)}_1$ for which no solution exists.
  This results in
  \begin{align*}
    T_\infty:= \{ && u^{(2)}_2 =0, \quad g^{(2)}_0 &= 0,                    &&                         && \\
                  && (kg^{(2)}_1-1)g^{(1)}_k+(-1)^k\frac{k!}{\ce^{k+1}} &=0   && \forall\ k\in\Z_{\ge0}, && \\
                  && kg^{(2)}_1-1 &\not=0                                   && \forall\ k\in\Z_{\ge1}  && \}\mbox{ .} \\
  \end{align*}

  Hence, $\sol_\ce(S)=\sol_\ce(T)\uplus \sol_\ce(T_\infty)$.

  For order $\ell=0$ this system has one solution $\{g^{(1)}_0=\frac{1}{\ce}, g^{(2)}_0=0\}$ and thus its differential counting polynomial is $1$.
  Its solution set is disjoint with that of $T$, which has differential counting polynomial $\infty^2-\infty$.
  Thus, the zeroth differential counting polynomial is $\infty^2-\infty+1$ for $\ell=0$.
  Now assume $\ell\ge1$.
  The only choice in the special case system $T_\infty$ is for $g^{(2)}_1$ and it may be chosen freely in $\C\setminus\left\{\frac{1}{k}\middle|k\in\Z_{\ge1}\right\}$.
  Thus, $c(T_\infty)=\infty-\aleph_0$.
  This implies that the counting sequence of $S$ is 
  \begin{align*}
    c(S)=l\mapsto
    \begin{cases}
      \infty^3-\infty^2+\infty-\aleph_0, & \ell\ge1\\
      \infty^2-\infty+1, & \ell=0\mbox{ .}
    \end{cases}
  \end{align*}
  These exceptional values for $g^{(2)}_1$ correspond to the indeterminate $\aleph_0$ in the differential counting polynomial.
  
  %
  %
  %

  All formal power series solutions of this example converge.
  This is implied for the ones of $T$ by Riquier's Existence Theorem \cite{Riquier}.
  For system $T_\infty$ the solutions of $u^{(2)}$ are lines and the radius of convergence for the formal power series solutions of $u^{(1)}$ is $|\ce|$ by the ratio test:
  \belowdisplayskip=-1em
  \begin{align*}
    \left|\frac{g^{(1)}_{k+1}}{(k+1)g^{(1)}_k}\right|=\left|\frac{kg^{(2)}_1-1}{(k+1)g^{(2)}_1-1}\right|\cdot\left|\frac{1}{\ce}\right|\longrightarrow\left|\frac{1}{\ce}\right|, k\to\infty
  \end{align*}
\end{example}

The following example demonstrates that the additional information contained in the differential counting polynomial can be used to decide that a symbolic solver of differential equations did not find all solutions.

\begin{example}\label{example_better_dim_poly}
  Let $U=\{u^{(1)},u^{(2)}\}$, $\Delta=\{\partial_t\}$, $F=\C$, and $<$ the orderly ranking with $u^{(1)}>u^{(2)}$.
  We show the following.
  For all $\ell\ge1$
  \begin{align*}
    \cd(S)=c(S)(\ell)=\infty^{\ell+2}-\infty^{\ell+1}+(\ell+1)\infty^\ell-\ell\infty^{\ell-1}
  \end{align*}
  for formal power series solutions of $S:=\{p:=u^{(2)}u^{(1)}_1-u^{(1)}=0\}$ centered around zero.
  The dimension polynomial is $\ell+2$ (using Theorem~\ref{theorem_counting_and_dimension} and the Low Power Theorem \cite[IV.\S15]{Kol}).
    
  Maple's \textsf{dsolve} \cite{maple17} returns an arbitrary $u^{(2)}(t)$ and
  \begin{align*}
    u^{(1)}(t)=a\cdot e^{\int_0^t\frac{1}{u^{(2)}(h)}\mathrm{d}h}
  \end{align*}
  for a constant $a$.
  This set of solutions depends on $\ell+2$ generically arbitrary constants up to order $\ell$, in  accordance with the dimension polynomial.
  The zeroth power series coefficient of $u^{(2)}(t)$ cannot be zero, as otherwise the integral does not exist.
  Thus, Maple's \textsf{dsolve} finds $\infty^{\ell+2}-\infty^{\ell+1}$ solutions up to order $\ell$ and a subset of the solutions with $\ell$-th counting polynomial $(\ell+1)\infty^\ell-\ell\infty^{\ell-1}$ is not found.
  The dimension polynomial does not account for these additional solutions, some of which are analytic.
  
  Now we show the claims from above.
  Use the ansatz $u^{(1)}(t)=\sum_{i=0}^\infty a_i\frac{t^i}{i!}$ and $u^{(2)}(t)=\sum_{i=0}^\infty b_i\frac{t^i}{i!}$.
  Let
  \begin{align*}
    \forget: \C\{U\}\to \C[a_i,b_i|i\in\Z_{\ge0}]: u^{(1)}_i\mapsto a_i, u^{(2)}_i\mapsto b_i\mbox{ .}
  \end{align*}
  
  Adding $b_0\not=0$ to $S$ yields $T:=\{p = 0, b_0 \not=0\}$ with $\ell$-th differential counting polynomial $c(T)(\ell)=(\infty-1)\infty^{\ell+1}$ for every order $\ell\ge1$.
  
  Complementary, the system $\{p=0, b_0=0\}$ is equivalent to $S_1:=\{ \partial_tp=0, a_0=0, b_0=0 \}$ by postponing $p$.
  It is part of the family
  \begin{align*}
    S_k:=\{ && \partial_t^kp = u^{(2)}u^{(1)}_{k+1}+(ku^{(2)}_1-1)u^{(1)}_k+{\textstyle\sum_{i=2}^k\binom{k}{i}u^{(2)}_iu^{(1)}_{k+1-i}} &= 0, \\
                              && a_0 = \ldots = a_{k-1} = b_0 &=0, \\
                              && {\textstyle\prod_{i=1}^{k-1}(ib_1-1)} & \not=0 & \}
  \end{align*}
  of systems.
  Consider the initial of the equation $\forget(\partial_t^kp)$, which is equal to $(kb_1-1)a_k$ after reduction in $S_k$.
  Adding its initial $(kb_1-1)\not=0$ to $S_k$, and postponing $\partial_t^kp$ results in the system $S_{k+1}$.
  Complementary, adding $(kb_1-1)=0$ to $S_k$ and postponing $\partial_t^kp$ yields the system
  \begin{align*}
  T_k:=\{ && \partial_t^{k+1}p &= 0, \\
          && a_0 = \ldots = a_{k-1} &=0, \\
          && b_0 = kb_1-1 &=0 \}.
  \end{align*}
  The inequations from $S_k$ are superfluous in $T_k$ because of the equation $kb_1-1=0$.
  The equation $\forget(\partial_t^{k+1+j}p)$ reduces to $\frac{1}{k}a_{k+1+j}+\binom{k+1+j}{2}b_2a_{k+j}$ in the context of $T_k$ for all $j\in\Z_{\ge0}$.
  This reduced form has the leader $a_{k+1+j}$ for all $j\in\Z_{\ge0}$; there is no constraint for $a_k$.

  Consider the remaining system $T_\infty:=\bigcup_{i=1}^\infty S_i$.
  The equations $a_k=0$ for all $k\ge0$ combine to the differential equation $u^{(1)}=0$; this makes the differential equation $p$ superfluous.  
  Furthermore, $b_1$ is not allowed to be of the form $\frac{1}{i}$ for any $i\in\Z_{\ge1}$.
  Summing up, the system $T_\infty:=\{u^{(1)}=0,b_0=0,ib_1\not=1\ \forall i\in\Z_{\ge1}\}$ describes these remaining solutions.

  We discuss the $\ell$-th differential counting polynomials for $\ell\ge1$ of the sets of solutions of $T_\infty$ and $T_k$, $k\ge1$.
  These systems have disjoint sets of solutions in orders $\ell\ge1$, since $b_1$ takes different values.
  The $\ell$-th differential counting polynomial of $T_k$ for $k\le\ell$ is $\infty^\ell$, since the values for the indeterminates $a_k,b_2,\ldots,b_\ell$ are freely chooseable and the other values are fixed.
  In the union $\biguplus_{k>\ell,k=\infty}\sol_0(T_k)_{\le\ell}$ the value for $b_1$ can be freely chosen except for the $\ell$ values $\frac{1}{1},\ldots,\frac{1}{\ell}$.
  Then, the indeterminates $b_2,\ldots,b_\ell$ have no constraint and the indeterminates $a_i$ are uniquely determined.
  Thus,
  
  \begin{align*}
    c\Big(\biguplus_{k>0,k=\infty}\sol_0(T_k)_{\le\ell}\Big)
    &=\sum_{\mathclap{1\le k\le\ell}}c\big(\sol_0(T_k)_{\le\ell}\big)
    +c\Big(\biguplus_{k>\ell,k=\infty}\sol_0(T_k)_{\le\ell}\Big).\\
    &=\ell\cdot \infty^\ell
    +(\infty-\ell)\cdot\infty^{\ell-1}\\
    &=(\ell+1)\infty^\ell-\ell\cdot\infty^{\ell-1}
  \end{align*}
  Adding this $\ell$-th counting polynomial to the one of $T$ results in $\infty^{\ell+2}-\infty^{\ell+1}+(\ell+1)\infty^\ell-\ell\infty^{\ell-1}$, as claimed above.

  Riquier's Existence Theorem \cite{Riquier} implies the convergence for the power series solutions of system $T$ for analytical initial conditions.
  System $T_\infty$ gives the zero power series for $u^{(1)}$, which converges and only restricts the choice for the first two power series coefficients of $u^{(2)}$, hence $u^{(2)}$ can be chosen to converge.
  The solutions of the systems $T_k$ can diverge even for analytical initial conditions.
  E.g., consider system $T_1$ and prescribe $b_0=0, b_1=1, b_2=1, b_i=0$ for all $i\ge3$.
  By the ratio test the radius of convergence of the solution for $u^{(1)}$ is zero, as 
  \begin{align*}
    \left|\frac{a_{k+1}}{(k+1)a_k}\right|
    =\frac{k-1}{k}\left|\frac{\sum_{i=2}^k\binom{k}{i}\frac{b_{i+1}}{i+1}a_{k+1-i}}{\sum_{i=2}^k\binom{k}{i}b_ia_{k+1-i}} +\frac{k}{2}b_2\right|
    =\frac{k-1}{2}
    \longrightarrow\infty
  \end{align*}
  for  $k\to\infty$.
  The analytical initial condition $b_0=0, b_1=1, b_2=0, b_i=i!$, for $i\ge3$, gives $1$ as radius of convergence by a similar computation.
\end{example}

\section{Proofs}\label{sect_proofs}

This section proves the theorems of the previous sections.

\subsection{Proof of Theorem~\ref{theorem_counting_well_defined}}\label{proof_theorem_counting_well_defined}

By abuse of notation, let $c$ denote the counting polynomials for subsets in $\C^k$ for all $1\le k\le n$.

\begin{proof}[Proof of existence]\label{proof_theorem_counting_well_defined_existence}
  Let $S\subset \C[y_1,\ldots,y_n]$ be a simple algebraic $\sigma$-sys\-tem.
  Furthermore, let $\tau(S_{y_i})$ be the degree of the equation if $S_{y_i}$ is a singleton of an equation, $\tau(S')=\prod_{p\in S'}\deg_{y_i}(p)$ if $S_{y_i}$ is a finite set of inequations, and $\tau(S_{y_i})=\infty-\aleph_0$ if $S_{y_i}$ is a countably infinite set of inequations.
  Then, the product $\prod_{i=1}^n \tau(S_{y_i})$ is a counting polynomial.
  The correctness of this formula follows from the fibration structure of simple systems as discussed in \cite{PleskenCounting}, which also holds for simple $\sigma$-systems.
\end{proof}

Write $T:=\sol(S)$ and $\overline{T}$ for its Zariski closure.
Let $\pi:\C^n\to \C^{n-1}$ be the projection to the first $n-1$ components.
The projected set $\pi(T)$ is equal to the solution set of the simple $\sigma$-system $S_{<y_n}$ in $\C[y_1,\ldots,y_{n-1}]$.

\begin{proof}[Proof of uniqueness]
  For systems (instead of $\sigma$-systems), the claim is shown in \cite[Prop.~3.3]{PleskenCounting}; in this case, Lazard's Lemma implies that the degree of the algebraic counting polynomial is equal to the dimension of the set of solutions.

  In this proof we can ignore sets of lower dimension, since any counting polynomial of such a set is of lower degree and we can proceed by an induction on $\dim(T)$.
  Any partition of $T$ into solution sets of algebraic $\sigma$-systems of the same dimension is finite.
  Such a finite partition does not change the degree and leading coefficient of the counting polynomial, by the same arguments as in step 3 of the proof of \cite[Prop.~3.3]{PleskenCounting}.
  Thus, in the following we can always assume that a set is suitably partitioned into a disjoint union of sets.
  
  The claim is clear for $n=1$.
  We show the claim for the dimension $n$ of the surrounding space under the assumption that it is shown for dimensions $0$ up to $n-1$.
  The crux of the proof is that only axiom \eqref{definition_counting_product} in Definition~\ref{definition_counting} allows to increase this dimension.

  By the assumption on $n-1$, the algebraic counting polynomials  of the two projections $\pi(T)\subseteq \pi(\overline{T})$ have the same degree, say $d$, and leading coefficients, say $a$, as their Zariski closures coincide.
  
  As first case consider that $S_{y_n}$ is a set of inequations.
  By Definition~\ref{definition_counting}.\eqref{definition_counting_product}, any algebraic counting polynomial of $T$ is an algebraic counting polynomial of $\pi(T)$ multiplied by $\infty-b$ for $b\in\Z[\aleph_0]$.
  In particular, any leading coefficient is $a$ and any degree is $d+1$.
  Furthermore, $\overline{T}=\pi(\overline{T})\times\C$ has a unique counting polynomial, which is $c(\pi(\overline{T}))\cdot\infty$ and also has leading coefficient $a$ and degree $d+1$.

  As second case consider that $S_{y_n}$ is an equation.
  In this case we do an induction over $\dim(T)$.
  The claim is clear for $\dim(T)=0$, so assume that it is shown for all dimensions from $0$ to $\dim(T)-1$.

  On the one hand, by Definition~\ref{definition_counting}.\eqref{definition_counting_product}, the degree of any counting polynomial of $T$ is again $d$ and any leading coefficient is $a\cdot\deg_{y_n}(S_{y_n})$.

  On the other hand, consider $\overline{T}$.
  The map $\pi$ makes $\overline{T}$ an $\deg_{y_n}(S_{y_n})$-sheeted cover of $\pi(\overline{T})$.
  Denote by $R\subseteq\pi(\overline{T})$ the corresponding set of ramification points and by $U:=\pi(\overline{T})\setminus R$ the set of unramified points.
  To apply Definition~\ref{definition_counting}.\eqref{definition_counting_product}, one needs to partition $\overline{T}$ into (a refinement of) $\pi^{-1}(R)$ and $\pi^{-1}(U)$, thus any algebraic counting polynomial needs to be defined using this partition.
  As $U$ and $R$ are locally closed, their algebraic counting polynomials exist and are unique.
  The Zariski closures of $U$ and $\pi(\overline{T})$ coincide, so by induction on $n$ the leading coefficient of $c(U)$ is $a$ and $\deg_\infty(c(U))=d$.
  By Definition~\ref{definition_counting}.\eqref{definition_counting_product}, the algebraic counting polynomial of $\pi^{-1}(U)$ has the same degree and leading coefficient as that one of $T$, as $\pi^{-1}(U)$ is an unramified $\deg_{y_n}(S_{y_n})$-sheeted cover of $U$ and the set $\pi^{-1}(R)$ is of lower dimension than $\overline{T}$.
\end{proof}

\subsection{Proof of Theorem~\ref{theorem_counting_subset}}\label{subsection_proof_theorem_counting_subset}

The proof of Theorem~\ref{theorem_counting_polynomials} is given in \cite[Cor.~3.4]{PleskenCounting}.
The following two lemmas directly generalize this proof to showing Theorem~\ref{theorem_counting_subset}.
We call a set $W\subset\C^n$ \textbf{elementarily well-fibered} if either $n=1$ and $W$ is constructible or $n>1$, $\pi(W)\subseteq\C^{n-1}$ is elementarily well-fibered, and all fibers of $\pi^{-1}(\{w\})$ for $w\in\pi(W)$ are constructible with equal algebraic counting polynomials.
They admit an algebraic counting polynomial in $\Z[\infty]$ by definition.

\begin{lemma}\label{lemma_well_fibred_fibres}
  Let $V$ be a well-fibered set.
  Then there exists a finite partition $V=\biguplus_{i=1}^kW_i$ of $V$ into elementarily well-fibered sets $W_i$.
\end{lemma}
\begin{proof}
  The claim clearly holds for $n=1$.
  The only one of the five axioms for the algebraic counting polynomial that allows one to increase the dimension is axiom \eqref{definition_counting_product}.
  In general, one needs to partition $V$ before applying axiom \eqref{definition_counting_product}, but this partition needs to be finite, as otherwise axiom \eqref{definition_counting_sum} is not applicable to recombine the resulting algebraic counting polynomials.
  Elementarily well-fibered sets are exactly the sets for which axiom \eqref{definition_counting_product} is applicable without additional partitioning.
\end{proof}

\begin{lemma}\label{lemma_well_fibred_unique}
  Let $V$ be a well-fibered set.
  Then the algebraic counting polynomial of $V$ is unique.
\end{lemma}
\begin{proof}
  The proof of \cite[Prop.~3.3]{PleskenCounting} regarding the uniqueness of algebraic counting polynomials holds for well-fibered sets.
  One only needs to replace a partition into solution sets of simple systems with a partition into elementarily well-fibered sets, which exists by Lemma~\ref{lemma_well_fibred_fibres}.
\end{proof}

\subsection{Proof of Theorem~\ref{theorem_sigma_describe_solutions}}\label{theorem_sigma_describe_solutions_proof}
  
  Transform the algebraically restricted system of differential equations $S$ by keeping all equations and inequations in $\C[G]$ and apply $\forget$ (cf.\ page~\pageref{forget}) to the differential equations and all their (iterated) derivatives.
  Call the resulting set $Q$; it consists of infinitely many equations and inequations in $\C[G]$ and has the same set of solutions as $S$.
  Write $G=\{\overline{g_1},\overline{g_2},\ldots\}$ ordered by the ranking, i.e., $\overline{g_i}<\overline{g_{i+1}}$ for all $i$.
  Note that $Q\cap \C[\overline{g_1},\ldots,\overline{g_i}]$ is finite for all $i\in\Z_{>0}$. 
  
  Let $\overline{g_j}$ be the largest element in $G$ of order $\ell$.
  Define the set $L_0:=\textsf{Decompose}(Q\cap \C[\overline{g_1},\ldots,\overline{g_j}])$ of simple systems, where \textsf{Decompose} is the Thomas decomposition algorithm from \cite{thomasalg_jsc}.
  Iteratively, define the sets $L_k$ of simple systems by making
  \begin{align*}
    \left\{\textsf{Decompose}(T\cup  (Q\setminus \C[\overline{g_1},\ldots,\overline{g_{j+k-1}}])\cap \C[\overline{g_1},\ldots,\overline{g_{j+k}}])\;\middle|\;T\in L_{k-1}\right\}
  \end{align*}
  comprehensive (cf.\ section~\ref{sect_simple_systems}) with respect to $\overline{g_j}$ for each $k\in\Z_{>0}$.
  Let $L_k':=\{T\cap \C[\overline{g_1},\ldots,\overline{g_j}]\mid T\in L_k\}$ for each $k\in\Z_{\ge0}$.
  The simple systems in $L_k'$ have disjoint sets of solutions, as $L_k$ is comprehensive w.r.t.\ $\overline{g_j}$.
  
  For any power series which is not a solution of the input system $S$ there exists a $k$ large enough such that it is no longer the solution of any system in $L_k$, as each constraint in $Q$ is taken into account at some step.
  
  Next we show that equations stabilize by looking at ideals.
  Let $J_k:=\bigcap_{T\in L_k'}\mcI(T)$ be an ideal in the Noetherian ring $\C[G]_{\le\ell}$ for each $k\in\Z_{\ge0}$.
  This ideal is equal to intersecting $\C[G]_{\le\ell}$ with the vanishing ideal of $Q\cap \C[\overline{g_1},\ldots,\overline{g_{j+k}}]$.
  In particular, this ascending chain of ideals does stabilize after a finite number $k'$ of steps.
  This stable ideal is the vanishing ideal of all power series solutions truncated at order $\ell$.
  
  By construction for each simple system $T_{k+1}\in L_{k+1}'$ there is a unique simple system $T_k\in L_k'$ with $\sol(T_{k+1})\subseteq \sol(T_k)$; if additionally $\mcI(T_{k+1})=\mcI(T_k)$, then call $T_{k+1}$ the (unique) heir of $T_k$.
  Define a successor as an element in the transitive hull of heir.
  
  As $J_{k'}$ is a radical ideal in the Noetherian ring $\C[G]_{\le\ell}$, it has a finite prime decomposition.
  There is a minimal $L_k''\subseteq L'_{k}$ such that $J_{k'}=\bigcap_{T\in L_k''}\mcI(T)$ for each $k\ge k'$.
  By increasing $k'$ we may assume that the cardinality of each $L_k''$ is equal for all $k\ge k'$.
  In particular, each $T_k\in L_k''$ has an heir in $L_{k+1}''$.
  
  A closer look at the algebraic Thomas decomposition algorithm \textsf{Decompose} reveals that a system and all its successors do not only have equal ideals but also equal sets of equations.
  In particular, a system and its heir only differ in their set of inequations.
  We can slightly adapt the algebraic Thomas decomposition algorithm \textsf{Decompose} such that the simple systems (and the candidate simple systems $S_T$) allow more than one inequation with the same leader, as long as the conditional gcd of these inequations with the same leader have no common zero with the system.
  This adaption changes nothing of the previous discussion.
  However, now the inequations of a simple system are a subset of the inequations of its heir, and thus the union of any system in $L_{k'}''$ with all its successors is a simple algebraic $\sigma$-system.
  
  This results in a finite set of algebraic $\sigma$-systems having truncated solutions that are dense in the truncated solutions of $S$.
  The complement of this dense set is described by a countable set of systems.
  Continue with these systems inductively.
  The ideals of these complementary systems are strictly larger than the previous ideals.
  In particular, descending chains of these systems are finite in length.
  Hence, the number of algebraic $\sigma$-systems remains countable.\qed

\subsection{Proof of Theorem~\ref{theorem_counting_and_dimension}}\label{proof_theorem_counting_and_dimension}

The claim for the $\ell$-th differential counting polynomial is a corollary to Theorem~\ref{theorem_counting_well_defined}.
Thereby, we can assume without loss of generality that the set of solutions $\sol_\ce(I)_{\le \ell}$ up to order $\ell$ is constructible.
Now, the claim follows directly from the definition of the differential dimension function $\Omega_I:\ell\mapsto\dim(F\{U\}_{\le \ell}/I_{\le \ell})$ and that the dimension coincides  with the degree of the algebraic counting polynomial.
The formula for the coefficient also follows from the proof of Theorem~\ref{theorem_counting_well_defined}; we do not need to consider the degrees of the (quasilinear) derivatives of the equations.
  
The claim for the differential counting polynomial follows, as the dimension polynomial $\omega_I$ ultimately coincides with $\Omega_I$.

\subsection{Proof of Theorem~\ref{theorem_without_inequations}}\label{proof_theorem_without_inequations}

We prove this theorem by creating suitable simple algebraic systems $S_{\le\ell}\subset\C[G]_{\le\ell}$, which describe the formal power series solutions of $S$ around a point $\ce$ truncated at order $\ell$.

For this, define $\overline{S}$ by applying $\forget$ (cf.\ page~\pageref{forget}) to the equations in $S$ and all their reductive prolongations (cf.\ \cite[\textsection3]{thomasalg_jsc}).
Then, define $S_{\le\ell}:=\overline{S}\cap\C[G]_{\le\ell}$.

It is straightforward that $S_{\le\ell}$ is a simple algebraic system in $\C[G]_{\le\ell}$, e.g.\ derivatives are squarefree, as they are quasilinear and the initial of a derivative is the separant of the original equation.

Next, we show that the formal power series solutions of $S$ around a point $\ce$ truncated at order $\ell$ are the same as those of $S$, i.e., $\sol_\ce(S)_{\le\ell}=\sol_\ce(S_{\le\ell})_{\le\ell}$.
This would be clear if $\overline{S}$ contained all derivatives of equations in $S$ and not only the reductive prolongations.
However, the non-reductive prolongations are redundant, as $S$ is involutive (cf.\ \cite[3.5]{thomasalg_jsc}).

Finally, the existence proof of Theorem~\ref{theorem_counting_well_defined} on page \pageref{proof_theorem_counting_well_defined_existence} allows to read off the counting polynomial from $S_{\le\ell}$:
The number of free variables is equal to the value of the differential dimension function at $\ell$  (and to the value of the dimension polynomial for $\ell$ large enough).
Furthermore, one needs to multiply the degrees of the equations to get the coefficient; these degrees are one for all derivatives, and thus one is left with the degrees of the equations in $S$.

The uniqueness follows from Theorem~\ref{theorem_differential_counting_subsets}, as no $\aleph_0$ appears.\qed

\section*{Acknowledgments}

The author was partly supported by the DFG Schwerpunkt SPP 1489 and Graduiertenkolleg Experimentelle und konstruktive Algebra of the DFG.
His gratitude goes to the anonymous referees for a thorough review with many valuable comments.

\bibliographystyle{plain}
\bibliography{DiffCountPoly.bbl}
  
\end{document}